\documentclass[11pt,reqno]{amsart}

\usepackage[utf8]{inputenc}
\usepackage{setspace}
\usepackage{geometry}
\usepackage{enumerate}
\usepackage{enumitem, xcolor, amssymb,latexsym,amsmath,bbm}
\usepackage{mathtools}
\usepackage[mathscr]{euscript}
\usepackage{amsmath}
\usepackage{amssymb}
\usepackage{enumitem} 
\newcommand{\R}{\mathbb{R}}

\newcommand{\N}{\mathbb{N}}

\newcommand{\sys}{\mathrm{sys}}

\usepackage{tikz}
\usepackage{wrapfig}
\usepackage{enumerate}
\usepackage{graphicx,float}
\usepackage{subfigure}

\usepackage{booktabs}

\usetikzlibrary{arrows.meta}

\usetikzlibrary{decorations.markings}
\usepackage[colorlinks=true,citecolor=blue, linkcolor=blue,urlcolor=blue]{hyperref}

\calclayout
\usepackage{thmtools}

\theoremstyle{plain}
\newtheorem{theorem}{Theorem}[section]

\newtheorem{lemma}[theorem]{Lemma}

\theoremstyle{definition}
\newtheorem{definition}[theorem]{Definition}
\newtheorem{question}[theorem]{Question}
\theoremstyle{remark}

\newtheorem{note}[theorem]{Note}

\numberwithin{equation}{section}
\numberwithin{figure}{section}
\allowdisplaybreaks

\title{Bolza-like surfaces in the Thurston set}

\author{Achintya Dey}
\address{
Department of Mathematics and Statistics\\ 
Indian Institute of Technology  \\ 
Kanpur, Uttar Pradesh-208016\\
India}
\email{achintd@iitk.ac.in}
\author{Bhola Nath Saha}
\address{
Department of Mathematics and Statistics\\ 
Indian Institute of Technology  \\ 
Kanpur, Uttar Pradesh-208016\\
India}
\email{sahabholanath497@gmail.com}
\author{Bidyut Sanki}
\address{
Department of Mathematics and Statistics\\ 
Indian Institute of Technology  \\ 
Kanpur, Uttar Pradesh-208016\\
India}
\email{bidyut@iitk.ac.in}
\date{\today}

\begin{document}



\begin{abstract}
   A surface in the Teichm\"uller space, where the systole function admits its maximum, is called a maximal surface. For genus two, a unique maximal surface exists, which is called the Bolza surface, whose systolic geodesics give a triangle decomposition of the surface. We define a surface as Bolza-like if its systolic geodesics decompose the surface into $(p, q, r)$-triangles for some integers $p,q,r$. In this article, we construct a one-parameter family of hyperbolic surfaces within the Thurston set $\chi_g$, which, in turn, yields Bolza-like surfaces for infinitely many genera $g\geq 9$. Next, we see an intriguing application of Bolza-like surfaces. In particular, we construct global maximal surfaces using these Bolza-like surfaces. Furthermore, we study a symmetric property satisfied by the systolic geodesics of our Bolza-like surfaces. We show that any simple closed geodesic intersects the systolic geodesics at an even number of points.
\end{abstract}
\maketitle
\tikzset{->-/.style={decoration={
  markings,
  mark=at position #1 with {\arrow{>}}},postaction={decorate}}}
  \tikzset{-<-/.style={decoration={
  markings,
  mark=at position #1 with {\arrow{<}}},postaction={decorate}}}
\section{Introduction}
Let $\mathcal{T}_g$ denote the Teichm\"uller space of a closed orientable surface $S_g$ of genus $g$. The systole function $\sys : \mathcal{T}_g \to \R, g \geq 2$, is defined on the Teichm\"uller space and for a hyperbolic surface $X\in \mathcal{T}_g$, it outputs the length of a shortest non-trivial geodesic in $X$. A closed geodesic in the surface $X$ having length $\sys(X)$ is called a \emph{systolic geodesic}. Hyperbolic surfaces with filling set of systolic geodesics are of great interest due to Thurston's well-known spine problem. A spine in the Teichm\"uller space $\mathcal{T}_g$ is a mapping class group equivariant deformation retract in $\mathcal{T}_g$ of minimal dimension. Thurston claimed that the set $\chi_g$ consisting of those hyperbolic surfaces with a filling set of systolic geodesics, called as Thurston set, is a spine in $\mathcal{T}_g$. He sketched a proof in \cite{thurstonspine} but unfortunately, the proof is incomplete. Even finding elements of the set $\chi_g$ is a challenging problem. Anderson-Parlier-Pettet (see Theorem 3 \cite{anderson2011small}) have shown that if the systolic geodesics fill a closed hyperbolic surface of genus $g$ then it contains at least $\pi \sqrt{g(g-1)}/\log{(4g-2)}$ systolic geodesics. The authors further constructed hyperbolic surfaces of genus $g$ with $2g+2$ systolic geodesics and showed that any $2g$ of these geodesics are sufficient to fill the surface. In \cite{anderson6146relative}, the authors constructed a sequence of surfaces $S_{g_k}$ in the Thurston set $\chi_{g_k}$ with large Bers constant, where $g_k$ is large enough. Later, B. Sanki \cite{MR3877282} has constructed hyperbolic surfaces with $2g$ systolic geodesics and he also introduces a systolic length increasing deformation in the moduli space. More recently, Bourque~\cite{bourque2022dimension}, Mathieu~\cite{mathieu2023estimating}) have studied the dimension of the set $\chi_g$.

Thurston's idea to deform $\mathcal{T}_g$ into a small neighborhood of $\chi_g$ is as follows (for more details, see \cite{thurstonspine}, Remark 4.4 of \cite{Lizhen}): Take a hyperbolic surface which is not an element of the set $\chi_g$. Consider a simple closed geodesic that does not intersect any of the systolic geodesics. Now, cut the surface along this curve and attach two funnels along these boundaries. Take a bi-infinite geodesic arc that intersects all the systolic geodesics and insert a thin strip along this arc. If we delete the funnels and reglue the boundaries, then we get a hyperbolic surface of the same genus and with a larger systolic length. Thus, by deforming a hyperbolic surface that is not a member of the Thurston set, we obtain a new surface with a larger systole. Motivated by Thurston's idea, in this paper, we study the following natural question.
\begin{question}
    Does there exist a systolic length increasing deformation on a surface $X\in\chi_g$?
\end{question}

 In this article, we introduce a deformation on hyperbolic surfaces in $\chi_g$ and show that this deformation is also systolic length increasing.

Akrout~\cite{Akrout} has shown that the systole function is a topological Morse function on the Teichm\"uller space. By collar lemma (see Theorem 4.1.1~\cite{Buser}), the function $\sys$ does not have a minima. The global maxima for the function $\sys$ are unknown except for genus $g=2$. Jenny \cite{Jenni} has shown that for a surface $X\in \mathcal{T}_2, $ $ \cosh(\sys(X)/2)\leq \sqrt{2}+1$ and equality occurs for a unique surface (up to isometry), namely the so-called Bolza surface. The Bolza surface has 12 systolic geodesics and they decompose the surface into triangles with each triangle of type $(4,4,4)$ (see Section 5, \cite{Schmutz}). Note that a hyperbolic triangle is called of type $(p,q,r)$, for some $p,q,r\in \N$ if its interior angles are $\pi/p, \pi/q$ and $\pi/r$. By Gauss-Bonnet theorem, we have $\frac{1}{p}+ \frac{1}{q}+ \frac{1}{r}<1$. Motivating from this, we define the following.
\begin{definition}
    A closed hyperbolic surface is called Bolza-like, if the systolic geodesics decompose it into $(p,q,r)$ triangles for some positive integers $p,q,r$.
\end{definition}
In this article, we show that our systolic length-increasing deformation on a surface in $\chi_g$ yields Bolza-like surfaces, for infinitely many genera $g\geq 9$. In particular, we prove the following theorem, which is the main result of the paper.

%In this article, we construct Bolza-like surfaces for infinitely many genera $g$.

\begin{restatable}{theorem}{mainthm}\label{thm:1.3}
For every integer $g=kn+1$, where $k,n\in \N, k\geq 2, n\geq 4$, there is a continuous function $\mathcal{D}_g:\left[ 0, \frac{\pi}{12}\right] \to \mathcal{T}_g$ satisfying the following.
    \begin{enumerate}
        \item $\mathcal{D}_g(t)\in \chi_g$, for all $t\in \left[ 0, \frac{\pi}{12}\right]$, where $\chi_g$ is the Thurston set.
        \item $\mathrm{sys}(\mathcal{D}_g(t_1))<\mathrm{sys}(\mathcal{D}_g(t_2))$, for all $0\leq t_1 < t_2 \leq \frac{\pi}{12}.$
        \item The surface $\mathcal{D}_g\left(\frac{\pi}{12}\right)$ has $(6g-6)$ many systolic geodesics that give a triangle decomposition of the surface. In particular, $\mathcal{D}_g\left(\frac{\pi}{12}\right)$ is Bolza-like.
    \end{enumerate}
\end{restatable}

The proof of Theorem~\ref{thm:1.3} consists of a rigorous construction of a one-parameter family of hyperbolic surfaces, together with an explicit analysis of the systolic geodesics on these surfaces.

Furthermore, we present an application (see Section \ref{Application}) of the Bolza-like surfaces constructed in Theorem~\ref{thm:1.3}. It is a well-known and difficult problem to find the critical points of the systole function as very few examples are known \cite{Schmutz, SchmutzGlobalMaximal, Ursula, an2025small}. In this article, using the Bolza-like surfaces, we construct examples of global maxima of the systole function for infinitely many genera.

% \textcolor{blue}{I believe the authors are noting that their construction from section 3 can be used to recover a result of Schmutz, by giving an explicit construction of global maximal surfaces for the systole function.}

% Next, we will see an interesting application of the Bolza-like surfaces that we obtain, namely the construction of global maxima of the systole function. \textcolor{blue}{It is always a difficult and interesting problem to find critical points of the systole function, as very few examples are known \cite{Schmutz, SchmutzGlobalMaximal, Ursula, an2025small}. Our construction }. 

Finally, we show that the distribution of the systolic geodesics in the Bolza-like surfaces is symmetric in the sense that any simple closed geodesic intersects the systolic geodesics at an even number of points. In particular, we prove the following theorem.

\begin{restatable}{theorem}{mainthmm}\label{thm:distribution}
    Any simple closed geodesic on the Bolza-like surfaces $\mathcal{D}_g\left(\frac{\pi}{12} \right)$, as obtained in Theorem \ref{thm:1.3}, intersects the collection $\mathrm{Sys}\left(\mathcal{D}_g\left(\frac{\pi}{12} \right)\right)$ of all systolic geodesics at an even number of points.
\end{restatable}

To prove Theorem \ref{thm:distribution}, we use the involution symmetry of the Bolza-like surface and analyze intersections via the quotient torus. A lift to the universal cover then shows that all intersections occur in pairs, and hence are even.

\section{Construction of Bolza-like surfaces}

In this section, we prove the main result (Theorem~\ref{thm:1.3}) of this article. Before going into the proof of Theorem~\ref{thm:1.3}, we develop two technical lemmas that are essential in proving the theorem.

\begin{lemma}\label{lemma1.2}
For $0\leq \epsilon<\frac{\pi}{4}$, there exists a unique hyperbolic quadrilateral with consecutive angles $\left(\frac{\pi}{4}-\epsilon\right),\left(\frac{\pi}{4}+\epsilon\right),\left(\frac{\pi}{4}-\epsilon\right) \text{ and }\left(\frac{\pi}{4}+\epsilon\right)$ and sides of equal length.
\end{lemma}
\begin{proof}
Consider a hyperbolic triangle with angles $(\pi/4+\epsilon), (\pi/8-\epsilon/2) \text{ and }(\pi/8-\epsilon/2)$. For the existence of such a triangle, see Theorem 7.16.2 of \cite{beardon2012geometry}. The uniqueness (up to isometry) of this triangle follows from the trigonometric identity given in Theorem 2.2.1$(ii)$ \cite{Buser}.

The quadrilateral is obtained by taking two disjoint isometric copies of such triangle and attaching them along the side opposite to angle $(\pi/4+\epsilon)$ (see Figure~\ref{quadrilateral}). 

\noindent\textbf{Uniqueness:} If such quadrilateral is cut along a diagonal, we get two isometric triangles. The uniqueness of the quadrilateral follows from the uniqueness of the triangles.
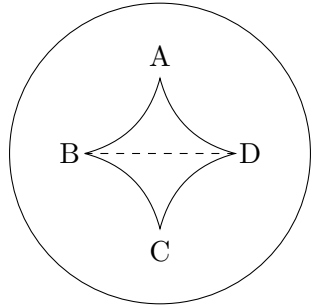
\begin{figure}[htbp]
\begin{center}
\begin{tikzpicture}[xscale=1,yscale=1]

\draw (0,0) circle (2);
\draw[dashed] (1,0)--(-1,0);
\draw (-1,0) [bend right] to (0,1) to (1,0) to (0,-1) to (-1,0); 
\draw (0,1.3) node{A} (-1.2,0) node{B} (0,-1.3) node{C} (1.2,0) node{D};

\end{tikzpicture}
\end{center}
 \caption{$\triangle ABD$ and $\triangle BDC$ are the isometric triangles with interior angles at each of the vertices $A$ and $C$ is $\frac{\pi}{4}-\epsilon$.}
\label{quadrilateral}
\end{figure}
\end{proof}

\noindent\textbf{Notation:} We call the quadrilateral in Lemma~\ref{lemma1.2} as \emph{square} and denote it by $Q_\epsilon$.

\begin{lemma}\label{lem:3.2}
    The common perpendiculars between the opposite sides of the square $Q_\epsilon$ pass through the point of intersection of the diagonals and they have the same length $\sqrt{2} \cos\epsilon +1$. 
\end{lemma}

\begin{proof}
    \begin{figure}[htbp]
        \centering
        \begin{tikzpicture}
            \draw [bend left=20] (-1.5,-1.5) to (1.5,-1.5) to (1.5,1.5) to (-1.5,1.5) to (-1.5,-1.5);
            \draw (-1.5,-1.5)--(1.5,1.5);
            \draw (1.5,-1.5)--(-1.5,1.5);
            \draw (0,-.35) node {$O$};
            \draw [dashed, thick, bend left=10] (-1.2,0) to (0,0) to (1.22,0);
            \draw (-1.6,0) node{$M$} (-1.75,1.5) node {$A$} (-1.75,-1.5) node {$B$} (1.5,0) node{$N$} (1.75,1.5) node {$D$} (1.75,-1.5) node {$C$};
        \end{tikzpicture}
        \caption{}
        \label{fig:3.2}
    \end{figure}
    We label the vertices of $Q_\epsilon$ by $A,B,C,$ and $D$ and let the interior angles at these vertices be $\frac{\pi}{4}+\epsilon$, $\frac{\pi}{4}-\epsilon$, $\frac{\pi}{4}+\epsilon$, and $\frac{\pi}{4}-\epsilon$, respectively. Let the diagonals of the square meet at $O$. Suppose that the perpendiculars from $O$ into the sides $AB$ and $CD$ meet at $M$ and $N$, respectively (see Figure \ref{fig:3.2}). We show that $MON$ is a geodesic, that is $\angle MON = \pi$ and this will prove that the common perpendicular between $AB$ and $CD$ passes through $O$. In triangles $\triangle AMO$ and $\triangle CNO$, $\angle MAO=\frac{\pi}{8}+ \frac{\epsilon}{2}=\angle NCO$, $\angle OMA=\frac{\pi}{2}= \angle ONC$, $OA=\frac{1}{2} AC= OC$. From the trigonometric identities for right angled triangle (see Theorem 2.2.2 \cite{Buser}), the triangles $\triangle AMO$ and $\triangle CNO$ are isometric and hence $\angle{AOM} =\angle{CON}$. This implies that $\angle MON= \pi$ because $AC$ is a geodesic. Now, we compute the length of the perpendicular $MN$. Using trigonometric identity $(iii)$ in Theorem 2.2.2 \cite{Buser} for the right-angled triangle $\triangle AMO$, we have
    \begin{align*}
        \sinh{OM}=\sinh{OA} \sin{\left(\frac{\pi}{8}+\frac{\epsilon}{2}\right)}.
    \end{align*}
    Therefore,
    \begin{align*}
        \cosh{MN}&=2\sinh^2{OM}+1\\
        &=2 \sinh^2{OA} \sin^2{\left(\frac{\pi}{8}+\frac{\epsilon}{2}\right)}+1\\
        &=(\cosh{AC}-1) \sin^2{\left(\frac{\pi}{8}+\frac{\epsilon}{2}\right)} +1\\
        &=\sqrt{2} \cos{\epsilon} +1.
    \end{align*}
    Similarly, one can check that the common perpendicular between the sides $AD$ and $BC$ passes through $O$ and its length is also given by $\sqrt{2} \cos{\epsilon} +1$.
\end{proof}
\begin{note}
    The distance between the opposite sides of the square $Q_\epsilon$ decreases with $\epsilon$.
\end{note}

%Here, we construct a family of hyperbolic surfaces $S_g(t)$ of genus $g$ (large enough) in the Thurston's set with the property that the systolic length of $S_g(t)$ increases with $t$. More precisely, we will prove the following.

\mainthm*

% \begin{repthm}[\ref{thm:1.3}]
%    For every integer $g=kn+1$, where $k,n\in \N, k\geq 2, n\geq 4$, there is a continuous function $\mathcal{D}_g:\left[ 0, \frac{\pi}{12}\right] \to \mathcal{M}_g$ satisfying the following.
%     \begin{enumerate}
%         \item $\mathcal{D}_g(t)\in \chi_g$, for all $t\in \left[ 0, \frac{\pi}{12}\right]$, where $\chi_g$ is the Thurston set.
%         \item $\mathrm{sys}(\mathcal{D}_g(t_1))<\mathrm{sys}(\mathcal{D}_g(t_2))$, for all $0\leq t_1 < t_2 \leq \frac{\pi}{12}.$
%         \item The surface $\mathcal{D}_g\left(\frac{\pi}{12}\right)$ has $(6g-6)$ many systolic geodesics that give a triangle decomposition of the surface. In other words, $\mathcal{D}_g\left(\frac{\pi}{12}\right)$ is Bolza-like.
%     \end{enumerate}
% \end{repthm}

\begin{proof}
    First, we construct a hyperbolic torus with $2g-2$ cone points with angles $\pi$ at each cone point as follows. For any two numbers $m,n\in \mathbb{N}$, take $mn$ identical copies of the squares $Q_\epsilon$ and arrange them in an $m\times n$ rectangular grid. We further identify the sides of the rectangles in a natural way so that we obtain a torus $T_\epsilon$ with $mn$ singular points and the angle at each singular point is $\pi$. If $mn$ is even, then there is a hyperbolic surface of the genus $g=\frac{mn+2}{2}$, say, $S_g(\epsilon)$ and an orientation preserving involution $\sigma_\epsilon$ of $S_g(\epsilon)$ with $2g-2$ fixed points such that the quotient of the surface by the involution is the torus $T_\epsilon$, that is $S_g(\epsilon)/\sigma_\epsilon\cong T_\epsilon$ (see Figure \ref{fig:involution}). We define the path $\mathcal{D}_g: \left[0,\frac{\pi}{4}\right) \to \mathcal{T}_g$ by
    $$\mathcal{D}_g(\epsilon)=S_g(\epsilon), \text{ for all } \epsilon \in [0,\pi/4).$$

    \begin{figure}[htbp]
        \centering
        \begin{tikzpicture}
            \draw (-1.2,2)..controls(-.5,6) and (.5,6)..(1.2,2);
            %\draw (-4,1)--(-2,1);
            \draw (-1.2,2)..controls(-1.3,1.5) and (-1.7,1)..(-2,1);
            %\draw (4,1)--(2,1);
            \draw (1.2,2)..controls(1.3,1.5) and (1.7,1)..(2,1);

            \draw (-1.2,-2)..controls(-.5,-6) and (.5,-6)..(1.2,-2);
            %\draw (-4,-1)--(-2,-1);
            \draw (-1.2,-2)..controls(-1.3,-1.5) and (-1.7,-1)..(-2,-1);
            %\draw (4,-1)--(2,-1);
            \draw (1.2,-2)..controls(1.3,-1.5) and (1.7,-1)..(2,-1);

            \draw (-2,1)..controls(-4,.7) and (-4,-.7)..(-2,-1);
            \draw (2,1)..controls(4,.7) and (4,-.7)..(2,-1);

            \draw (0.05,.95) to [bend left=60] (0.05,2.05);
            \draw (-.08,1.1) to [bend right=60] (-.08,1.9);

            \draw (0.05,2.95) to [bend left=60] (0.05,4.05);
            \draw (-.08,3.1) to [bend right=60] (-.08,3.9);

            \draw (0.05,-.95) to [bend right=60] (0.05,-2.05);
            \draw (-.08,-1.1) to [bend left=60] (-.08,-1.9);

            \draw (0.05,-2.95) to [bend right=60] (0.05,-4.05);
            \draw (-.08,-3.1) to [bend left=60] (-.08,-3.9);

            %\draw [very thick] (-.1,0) node {$\vdots$};
            \foreach \y in {0,0.5,-.5}
            {
                \fill (-0.1,\y) circle (1.5pt);
            }

            \draw [dotted] (0,6)--(0,-6);

            % \draw (1.45,.05) to [bend right=60] (2.55,.05);
            % \draw (1.55,.0) to [bend left=60] (2.45,.0);
            
            \draw (1.45,0.05) to [bend right=60] (2.55,0.05);
            \draw (1.6,-.08) to [bend left=60] (2.4,-.08);

            \draw (-1.45,0.05) to [bend left=60] (-2.55,0.05);
            \draw (-1.6,-.08) to [bend right=60] (-2.4,-.08);

            % \draw (-2.05,.95-1.5) to [bend left=60] (-2.05,2.05-1.5);
            % \draw (1.92-4.1,1.1-1.5) to [bend right=60] (1.92-4.1,1.9-1.5);

            \draw [->-=1] (-.3,5.5) arc
            [
                start angle=100,
                end angle=420,
                x radius=1cm,
                y radius =.12cm
            ] ;
        
            \draw [->-=1] (4,0)--(5,0);
            \draw (1.1,5.2) node {$\sigma_\epsilon$};

            \draw (-1.2+9,2)..controls(-1.3+9,1.5) and (-1.7+9,1)..(-2+9,1);
            \draw (-1.2+9,-2)..controls(-1.3+9,-1.5) and (-1.7+9,-1)..(-2+9,-1);
            \draw (-2+9,1)..controls(-4+9,.7) and (-4+9,-.7)..(-2+9,-1);

            \draw (-1.45+9,0.05) to [bend left=60] (-2.55+9,0.05);
            \draw (-1.6+9,-.08) to [bend right=60] (-2.4+9,-.08);
            %\draw (5.5,0) node {$\bullet$};
            \draw (9,5)..controls (8.8,5) and (8.5,4.2)..(-1.2+9,2);
            \draw (9,-5)..controls (8.8,-5) and (8.5,-4.2)..(-1.2+9,-2);

            \draw (0.05+9,.95) to [bend left=60] (0.05+9,2.05);
            \draw (0.05+9,2.95) to [bend left=60] (0.05+9,4.05);
            \draw (0.05+9,-.95) to [bend right=60] (0.05+9,-2.05);
            \draw (0.05+9,-2.95) to [bend right=60] (0.05+9,-4.05);

            \draw (9,5) to [bend right=10] (0.05+9,4.05);
            \draw (0.05+9,2.95) to [bend right=10] (0.05+9,2.05);
            \draw [very thick, dotted] (0.05+9,.95) to [bend right=10] (0.05+9,-.95);
            \draw (9,-5) to [bend left=10] (0.05+9,-4.05);
            \draw (0.05+9,-2.95) to [bend left=10] (0.05+9,-2.05);
        \end{tikzpicture}
        \caption{}
        \label{fig:involution}
    \end{figure}

    Now, we investigate all the systolic geodesic of the surface $S_g(\epsilon)$. As the involution $\sigma_\epsilon$ acts as a rotation by the angle $\pi$, any geodesic segment in $T_\epsilon$ joining two distinct singular points (not passing through any other singular points) lift to a geodesic in $S_g(\epsilon)$. In particular, the sides of squares and the diagonals of the square joining the vertices with interior angle $\left (\frac{\pi}{4}+\epsilon\right)$ lift to geodesic in $S_g(\epsilon)$. We denote the length of these segments by $l_\epsilon$ and $x_\epsilon$, respectively (see Figure \ref{FIG:0.1}). We show that $\sys(S_g(\epsilon)) =\min\{2l_\epsilon,2x_\epsilon\}$ and the systolic geodesics are either the lifts of the sides of $Q_\epsilon$ or the lifts the of diagonals or both (depending on $\epsilon$). Using hyperbolic trigonometric identity $(ii)$ in Theorem 2.2.1 \cite{Buser}, we have
    \begin{align}
        \cosh\left(l_\epsilon\right)&=\frac{\cos\left(\frac{\pi}{8}-\frac{\epsilon}{2}\right)+\cos\left(\frac{\pi}{8}-\frac{\epsilon}{2}\right)\cos\left(\frac{\pi}{4}+\epsilon\right)}{\sin\left(\frac{\pi}{8}-\frac{\epsilon}{2}\right)\sin\left(\frac{\pi}{4}+\epsilon\right)} \label{eqn: 3.1}, \text{ and }\\
        \cosh (x_\epsilon)&= \frac{\cos\left( \frac{\pi}{4}-\epsilon\right)+ \cos^2\left( \frac{\pi}{8}+\frac{\epsilon}{2}\right)}{\sin^2\left( \frac{\pi}{8}+\frac{\epsilon}{2}\right)}\label{eqn:3.2}.
    \end{align}

    Initially, when $\epsilon=0$, $l_0\approx 2.448 <3.057 \approx x_0$. Differentiating equations \eqref{eqn: 3.1} and \eqref{eqn:3.2} with respect to $\epsilon$, we have
    \begin{align*}
        \frac{d}{d\epsilon}\cosh\left(l_\epsilon\right)&=\frac{1}{4}\cdot \frac{\sin \left(\frac{\pi}{4}+\epsilon\right)-\sin \left(\frac{\pi}{4}- \epsilon\right)}{\sin^2 \left(\frac{\pi}{8}-\frac{\epsilon}{2}\right )\sin^2 \left(\frac{\pi}{4}+\epsilon\right )}>0,  \ \text{ for } 0\leq \epsilon < \frac{\pi}{4}, \text{ and }\\
        \frac{d}{d\epsilon}\cosh({x_\epsilon})&= -\frac{\sqrt{2} \cdot\cos{\left( \frac{\pi}{8}-\frac{\epsilon}{2}\right)}}{\sin^3{\left( \frac{\pi}{8}+\frac{\epsilon}{2}\right)}}<0 \text{ for } 0\leq \epsilon<\frac{\pi}{4}.
    \end{align*}
    Thus, $l_\epsilon$ is strictly increasing and $x_\epsilon$ strictly decreasing in $0\leq \epsilon<\frac{\pi}{4}$ with $x_\epsilon\to 0$ as $\epsilon \to \frac{\pi}{4}$. Therefore, there exists some $\epsilon_0\in (0, \pi/4)$ such that $l_{\epsilon_0}=x_{\epsilon_0}$. But in this case 
    $$\frac{\pi}{8}+\frac{\epsilon_0}{2}=\frac{\pi}{4}-\epsilon_0 \implies \epsilon_0=\frac{\pi}{12}.$$

    Now, we show that for any simple closed geodesic $\gamma$ in $S_g(\epsilon)$, $l(\gamma) \geq \min \{2l_\epsilon,2x_\epsilon\}$, for $0\leq \epsilon \leq \frac{\pi}{12}$. There are following two cases to consider.

    \textbf{Case 1:} $\gamma$ passes through a fixed point of the involution $\sigma_\epsilon$. Then $\gamma$ must pass through another fixed point that is diametrically opposite to the first one. In the torus $T_\epsilon$, the shortest arc joining two singular points has lengths either $l_\epsilon$ or $x_\epsilon$ (depending on $\epsilon$). Thus, $l(\gamma)\geq \min \{2l_\epsilon,2x_\epsilon\}$.

    \textbf{Case 2:} $\gamma$ does not pass through any fixed points. Then $\gamma$ project onto a closed curve (possibly non-simple) in $T_\epsilon$. In this case, we choose $m,n$ large enough such that $l(\gamma)\geq \min \{2l_\epsilon,2x_\epsilon\}$. 
    
    Now, we find lower bounds of $m,n$, and consequently of the genus $g$. For $0 \le \epsilon < \frac{\pi}{12}$, we have $l_{\epsilon} < x_{\epsilon}$, and $l_{\epsilon}$ attains its maximum value at $\epsilon = \frac{\pi}{12}$, where $l_{\pi/12} = x_{\pi/12}\approx 2.55$. By Lemma \ref{lem:3.2}, the shortest distance $y_\epsilon$ between the opposite sides is decreasing and $y_{\pi/12} \approx 1.506$. As $4$ is the smallest integer such that $4\times y_{\pi/12}> 2\times l_{\pi/12}$, we have $m,n \geq 4$. Therefore, the genus $g=\frac{mn+2}{2} \geq 9$. As $mn$ is also even, we take $m=2k, k\geq 2$ and hence $g=kn+1$, where $k,n\in \N, k\geq 2, n\geq 4$.

    Now we show that $S_g(\epsilon)\in \chi_g$ for $0\leq \epsilon \leq \pi/12$. As $l_\epsilon< x_\epsilon$ for $0\leq \epsilon < \pi/12$, the systolic geodesics decompose $S_g(\epsilon)$ into the squares $Q_\epsilon$. For $\epsilon= \pi/12$, the diagonals further decompose it into triangles with each interior angle $\pi/6$. This completes the proof.

    \begin{figure}[htbp]
    \begin{center}
    \begin{tikzpicture}
    
     \draw[step=1.5 cm] (0,0) grid (7.5,7.5) ;
     \draw[thin]  (1.5,6) circle (.4 cm);
     \draw (.8,6.4) node{$\frac{\pi}{4}-\epsilon$} (2.25,6.5) node{$\frac{\pi}{4}+\epsilon$} (2.25,5.5) node{$\frac{\pi}{4}-\epsilon$} (.8,5.5) node{$\frac{\pi}{4}+\epsilon$};
     \draw (3,4.5)--(4.5,3);
     \draw [red] (3,4.1) arc
    [
        start angle=270,
        end angle=315,
        x radius=.4cm,
        y radius =.4cm
    ] ;
    \draw [red] (4.1,3) arc
    [
        start angle=180,
        end angle=135,
        x radius=.4cm,
        y radius =.4cm
    ] ;
    \draw [red] (3.4,3) arc
    [
        start angle=0,
        end angle=90,
        x radius=.4cm,
        y radius =.4cm
    ] ;
     \draw (2.5,4) node {\tiny$\frac{\pi}{8}+\frac{\epsilon}{2}$} (2.5,3.3) node {\tiny$\frac{\pi}{4}-\epsilon$} (4.1,2.7) node {\tiny$\frac{\pi}{8}+\frac{\epsilon}{2}$};
     \draw[thick, cyan] (3.9,3.9) node{$x_\epsilon$};
     \draw[thick, cyan] (4.7,3.9) node{$l_\epsilon$};
     \end{tikzpicture}
    \end{center}
    \caption{The torus $T_\epsilon$}
    \label{FIG:0.1}
    \end{figure}
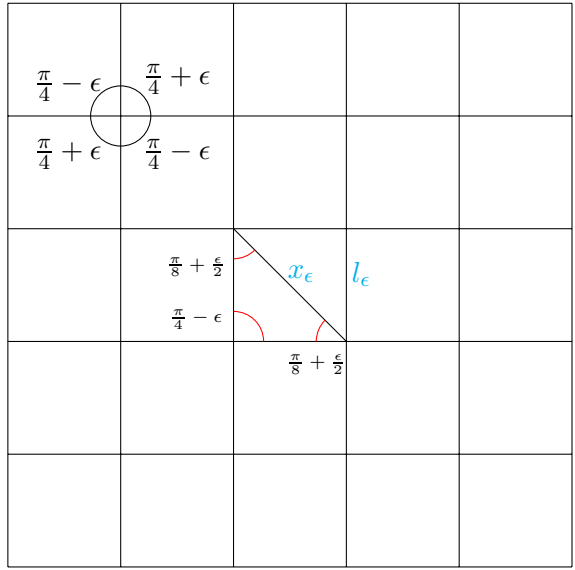
\end{proof}

\begin{note}
    The set of systoles on the surface $S_g\left(\frac{\pi}{12}\right)$ forms a distance Delaunay triangulation with $2g-2$ vertices (for a detailed discussion on distance Delaunay triangulation, see Section 2, \cite{MinimalDelaunay}).
\end{note}

\begin{note}
    Though the systolic geodesics in $S_g\!\left(\frac{\pi}{12}\right)$ decompose the surface into triangles, similarly to the Bolza surface, the surface is still not maximal. Indeed, a maximal surface of genus $g$ must contain at least $6g - 5$ systolic geodesics (see Theorem 2.8 in \cite{Schmutz}), whereas these Bolza-like surfaces have only $6g - 6$ systoles. Recall that a maximal surface is one for which the systole function attains a (local or global) maximum.
\end{note}

\section{An application of Bolza-like surfaces}\label{Application}
In this section, we show that Bolza-like surfaces provide a direct application of the construction developed in Section 2, enabling us to recover a result of Schmutz by explicitly constructing globally maximal surfaces for the systole function.

\textbf{A construction of global maximal surfaces:} Let $X$ be a closed orientable hyperbolic surface which admits a tessellation by triangles of equal side lengths and the number $N$ of triangles around each vertex is the same. For $a>0$, let $H_a$ denote a right-angled hyperbolic hexagon with lengths of each alternating side $a$. Then the surface $X(a)$ which has the same tessellation structure as $X$ in which the triangles are replaced by the hexagons $H_a$ and the triangle vertices by boundary geodesics of equal length $Na$ (see Figure \ref{fig:localPic}). Then $X(0)$ is the surface where the boundaries become cusps. In Section 3.1 \cite{SchmutzGlobalMaximal}, P. Schmutz proved that the surface $X(0)$ is a global maximal surface with a trivial automorphism group.

\subsection*{Global maximal surfaces from Bolza-like surfaces}
In the Bolza-like surface $\mathcal{D}_g\left( \frac{\pi}{12}\right)$, as obtained in Theorem \ref{thm:1.3}, the systolic geodesics gives a triangle decomposition with equal side lengths of the triangle and the number of vertices of the triangles is $2g-2$. Therefore, we will be able to find the global maximal surface in the Teichm\"uller space of genus $g$ and with $2g-2$ punctures, where $g=kn+1$, $k,n\in \N, k\geq 2, n\geq 4$. 

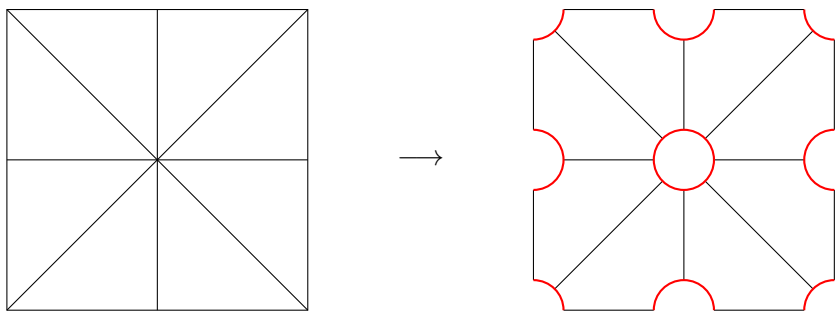
\begin{figure}[htbp]
    \centering
    \begin{tikzpicture}
        \draw (-2,-2)--(2,-2)--(2,2)--(-2,2)--cycle;
        \draw (-2,0)--(2,0) ;
        \draw (0,2)--(0,-2) ;
        \draw (-2,-2)--(2,2) ;
        \draw (2,-2)--(-2,2) ;

        \draw (3+2,1.6)--(3+2,.4);
        \draw (3+2,-1.6)--(3+2,-.4);
        \draw (7+2,1.6)--(7+2,.4);
        \draw (7+2,-1.6)--(7+2,-.4);
        \draw (3.4+2,2)--(4.6+2,2);
        \draw (5.4+2,2)--(6.6+2,2);
        \draw (3.4+2,-2)--(4.6+2,-2);
        \draw (5.4+2,-2)--(6.6+2,-2);
        \draw (3.4+2,0)--(4.6+2,0);
        \draw (5.4+2,0)--(6.6+2,0);
        \draw (5+2,.4)--(5+2,1.6);
        \draw (5+2,-.4)--(5+2,-1.6);
        \draw ({3+.4*cos(315)+2},{2+.4*sin(315)})--({5+.4*cos(135)+2},{.4*sin(135)});
        \draw ({5+.4*cos(315)+2},{.4*sin(315)})--({7+.4*cos(135)+2},{-2+.4*sin(135)});
        \draw ({7+.4*cos(225)+2},{2+.4*sin(225)})--({5+.4*cos(45)+2},{.4*sin(45)});
        \draw ({5+.4*cos(225)+2},{.4*sin(225)})--({3+.4*cos(45)+2},{-2+.4*sin(45)});

        \draw [thick, red] (5+2,0) circle(.4 cm);
        \draw[thick, red] (3+2,1.6) arc[start angle=270, end angle=360, radius=.4];
        \draw[thick, red] (6.6+2,2) arc[start angle=180, end angle=270, radius=.4];
        \draw[thick, red] (3+2,-1.6) arc[start angle=90, end angle=0, radius=.4];
        \draw[thick, red] (6.6+2,-2) arc[start angle=180, end angle=90, radius=.4];
        \draw[thick, red] (3+2,-.4) arc[start angle=270, end angle=450, radius=.4];
        \draw[thick, red] (7+2,.4) arc[start angle=90, end angle=270, radius=.4];
        \draw[thick, red] (4.6+2,2) arc[start angle=180, end angle=360, radius=.4];
        \draw[thick, red] (4.6+2,-2) arc[start angle=180, end angle=0, radius=.4];
        
        \draw (3.5,0) node {$\longrightarrow$};
    \end{tikzpicture}
    \caption{Local Picture around a vertex}
    \label{fig:localPic}
\end{figure}

\section{Distribution of the systolic geodesics on the Bolza-like surfaces}
In this section, we show that the distribution of the systolic geodesics on the Bolza-like surfaces is symmetric in the theorem below.
\mainthmm*

\begin{proof}
    Let $\gamma$ be a simple closed non-systolic geodesic in $\mathcal{D}_g\left(\frac{\pi}{12} \right)$. Then the following two cases may arise.

    \noindent \textbf{Case 1:} Let $\gamma$ pass through some fixed points of the involution $\sigma_\epsilon$ (as discussed in Theorem \ref{thm:1.3}). Then $\gamma$ is preserved by the involution $\sigma_\epsilon$ and it passes through an even number of fixed points of $\sigma_\epsilon
    $. Hence, if $\gamma$ intersects a systolic geodesic at a point, another point must exist at which $\gamma$ intersects the same geodesic. This point is the image of the first point of intersection under the involution. Thus, the theorem holds in this case.

    \noindent \textbf{Case 2:} Let $\gamma$ do not pass through the fixed points of the involution $\sigma$. Then $\gamma$ descends to a closed geodesic $\overline{\gamma}$ in the torus $T_{\frac{\pi}{12}}$ which is not necessarily simple. As the parity of the geometric intersection number depends only on the topology of the surface, for convenience, we may consider $T=T_{\frac{\pi}{12}}$ as a topological surface. Let $T$ be obtained by identifying the opposite sides of the unit Euclidean square (see Figure \ref{fig torus}). Then the following two subcases may arise. 
    \begin{figure}[H]
        \centering
        \begin{tikzpicture}
            \draw[step=1cm,gray,very thin] (0,0) grid (5,4);
            % \foreach \x in {0,1,...,4} 
            % {
            % \draw[gray, very thin] (0,4-\x) -- (\x,4);
            % \draw[gray, very thin] (\x+1, 0) -- (5, 4-\x);
            % }
            \foreach \x in {0,1,...,4} 
            {
            \draw[gray, very thin] (0,\x) -- (\x,0);
            \draw[gray, very thin] (\x+1, 4) -- (5, \x);
            }
            % \draw (-.5,-.2) node {(0,0)};
            % \draw (5.5,4.1) node {(1,1)};
        \end{tikzpicture}
        \caption{Flat torus taking $m=4,n=5$}
        \label{fig torus}
    \end{figure}
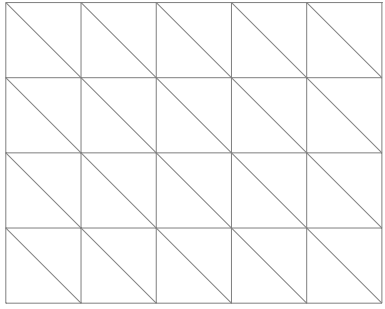

    \noindent \textbf{Subcase 1:} $\overline{\gamma}$ is simple. Then $\overline{\gamma}$ is homotopic to some $(p,q)$ curve $\delta$ in the torus $T$ where $p,q$ are integers with $\gcd(p,q)=1$. Note that the number of points of intersection of $\overline{\gamma}$ with the systolic geodesics is the same as the number of points of intersection of $\delta$ and the curves triangulating the torus. Now, consider everything in the universal cover $\R^2$ (see Figure \ref{fig univer1}). Let $\Tilde{\delta}$ be the lift of $\delta$ in the universal cover. Then the number of points of intersection of $\Tilde{\delta}$ with the lifts of the triangulating curves is the same as the number of points of intersection of the red arc in Figure \ref{fig univer1} with the lifts of the triangulating curves. Now, we count the number of points of intersection of the red arc. First, suppose that $pq>0$. The number of points of intersection of the horizontal part of the red arc with the triangulating curves is $2pn+1$ and for the vertical part is $2qm-1$ (1 is subtracted because one intersection point is already counted). Hence, the total number is $2(pn+qm)$, which is even.

    Now, suppose that $pq < 0$. Without loss of generality, assume that $p<0$. In this case, the number of points of intersection of the horizontal part is $2|p|n-1$, and for the vertical part is $2qn-1$ (see Figure \ref{fig univer2}). Hence the total is $2(|p|n+qn-1)$, which is even.

    \begin{figure}[htbp]
        \centering
        \begin{tikzpicture}
            \fill[gray!20!white]  (0,0) rectangle (5,4);
            \fill[gray!20!white]  (6,5) rectangle (11,7);
            \draw[step=1cm,gray,very thin] (-.9,-.9) grid (11.9,6.9);
            \foreach \x in {0,1,...,6} 
            {
            \draw[gray, very thin] (-1,\x+1) -- (\x+1,-1);
            }
            \foreach \x in {0,1,...,3} 
            {
            \draw[gray, very thin] (\x,7) -- (\x+8,-1);
            }
            \foreach \x in {0,1,...,6} 
            {
            \draw[gray, very thin] (\x+4,7) -- (12,\x-1);
            }
            \draw [red] (0,.3)--(5.4,.3)--(5.4,5.4)--(6,5.4);
            \draw [cyan] (0,0) rectangle (5,4);
            \draw [cyan] (6,7)--(6,5)--(11,5)--(11,7);

            \draw (0,.3) node {\tiny $\bullet$} (6,5.4) node {\tiny $\bullet$} (6,5) node {\tiny $\bullet$} (0,0) node {\tiny $\bullet$};
            \fill[white!50!white] (5.1+1,4.5) rectangle (5.9+1,4.9); 
            \draw (-.5,-.3) node {$(0,0)$} (5.5+1,4.75) node {$(p,q)$};
        \end{tikzpicture}
        \caption{Tessellation in the universal cover; the shaded rectangle represents a fundamental domain}
        \label{fig univer1}
    \end{figure}

    \begin{figure}[htbp]
        \centering
        \begin{tikzpicture}
            
            \fill[gray!20!white]  (6,0) rectangle (11,4);
            \fill[gray!20!white]  (0,5) rectangle (5,7);
            \draw[step=1cm,gray,very thin] (-.9,-.9) grid (11.9,6.9);
            \foreach \x in {0,1,...,6} 
            {
            \draw[gray, very thin] (-1,\x+1) -- (\x+1,-1);
            }
            \foreach \x in {0,1,...,3}
            {
            \draw[gray, very thin] (\x,7) -- (\x+8,-1);
            }
            \foreach \x in {0,1,...,6} 
            {
            \draw [gray, very thin] (\x+4,7) -- (12,\x-1);
            }
            \draw [red] (6,.75)--(0.4,.75)--(0.4,5.5)--(0,5.5);
            \draw [cyan] (6,0) rectangle (11,4);
            \draw [cyan] (0,7)--(0,5)--(5,5)--(5,7);

            \draw (6,.75) node {\tiny $\bullet$} (0,5.5) node {\tiny $\bullet$} (0,5) node {\tiny $\bullet$} (6,0) node {\tiny $\bullet$};
            \fill[white!50!white] (-1.1+1,4.5) rectangle (-1.9+1,4.9);
            \fill[white!50!white] (-1.1+1+6,4.5-5) rectangle (-1.9+1+6,4.9-5);
            \draw (-.5+6,-.3) node {$(0,0)$} (5.5+1-7,4.75) node {$(p,q)$};

            %\draw[thick] plot[smooth cycle, tension=1] coordinates {(0,0) (2,1) (3,-1) (1,-3) (-1,-1)};
             %\draw[thick, domain=0:4, samples=100] plot (\x, {sin(90*\x)});
        \end{tikzpicture}
        \caption{Tessellation in the universal cover; the shaded rectangle represents a fundamental domain}
        \label{fig univer2}
    \end{figure}
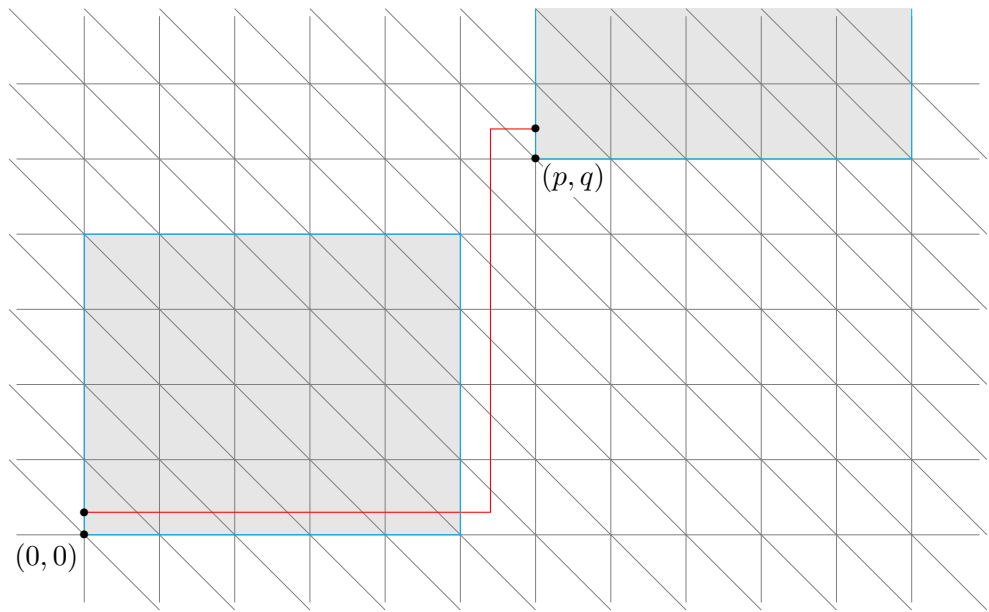

\noindent\textbf{Subcase 2:} $\overline{\gamma}$ is not simple. Take a lift of $\overline{\gamma}$ on the universal cover $\R^2$ (see Figure \ref{fig univer3}). The red loop (or loops in general) intersects each triangulating curve at an even number (0 or 2) of points. The remaining portion produces a simple closed curve on the torus and by Subcase 1, this part also intersects the triangulating curve at an even number of points. This completes the proof.

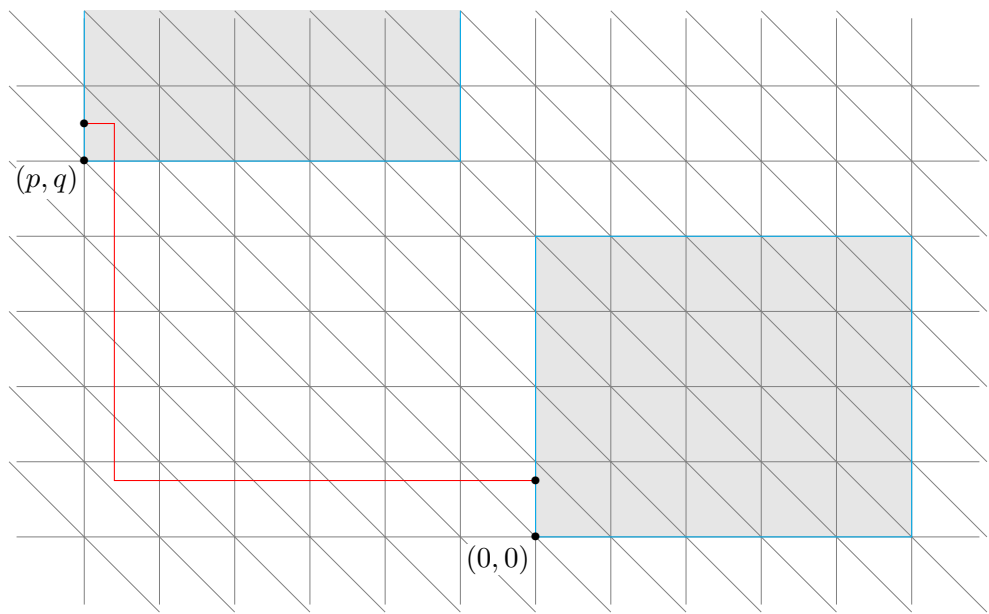
\begin{figure}[htbp]
        \centering
        \begin{tikzpicture}
            \fill[gray!20!white]  (0,0) rectangle (5,4);
            \fill[gray!20!white]  (6,5) rectangle (11,7);
            \draw[step=1cm,gray,very thin] (-.9,-.9) grid (11.9,6.9);
            \foreach \x in {0,1,...,6} 
            {
            \draw[gray, very thin] (-1,\x+1) -- (\x+1,-1);
            }
            \foreach \x in {0,1,...,3} 
            {
            \draw[gray, very thin] (\x,7) -- (\x+8,-1);
            }
            \foreach \x in {0,1,...,6} 
            {
            \draw[gray, very thin] (\x+4,7) -- (12,\x-1);
            }
           \draw (0,.3) node {\tiny $\bullet$} (6,5.4) node {\tiny $\bullet$};
           \draw [red] (2.2,1.5)..controls(4.5,5) and (-1,3)..(2.2,1.5);
           \draw [cyan] (0,.3)..controls(1,.2) and (2,1.2)..(2.2,1.5);
           \draw [cyan] (2.2,1.5)..controls(4,1) and (5.5,5)..(6,5.4);
           \fill[white!50!white] (5.1+1,4.5) rectangle (5.9+1,4.9);
           \draw (-.5,-.3) node {$(0,0)$} (5.5+1,4.75) node {$(p,q)$};
           % \draw[thick] plot[smooth cycle, tension=1] coordinates {(2.2,1.7) (2.4,2.8) (1.5,3.5) (.5,2.5) (1,1.5)};
             %\draw[thick, domain=0:4, samples=100] plot (\x, {sin(90*\x)});
        \end{tikzpicture}
        \caption{Tessellation in the universal cover; the shaded rectangle represents a fundamental domain}
        \label{fig univer3}
    \end{figure}
    
\end{proof}

\section*{Acknowledgements}

The authors would like to thank the anonymous referee for useful suggestions and comments in improving the exposition.

\bibliographystyle{alpha}
\bibliography{bibliography.bib}

\end{document}